\documentclass[12pt]{amsart}
\usepackage{amsmath,amsthm,amsfonts,amssymb}

\newtheorem{metatheorem}{metatheorem}[section]
\newtheorem{theorem}[metatheorem]{Theorem}

\newtheorem{lemma}[metatheorem]{Lemma}
\newtheorem{corollary}[metatheorem]{Corollary}

\newtheorem{itdefinition}[metatheorem]{Definition}
\newtheorem{itexample}[metatheorem]{Example}
\newtheorem{itremark}[metatheorem]{Remark}
\newtheorem{itquestion}[metatheorem]{Question}
\newtheorem{itproblem}[metatheorem]{Problem}

\newenvironment{example}%
    {\begin{itexample}\begin{rm}}{\end{rm}\end{itexample}}
\newenvironment{definition}%
    {\begin{itdefinition}\begin{rm}}{\end{rm}\end{itdefinition}}
\newenvironment{remark}%
    {\begin{itremark}\begin{rm}}{\end{rm}\end{itremark}}
    {\begin{itquestion}\begin{rm}}{\end{rm}\end{itquestion}}
    {\begin{itproblem}\begin{rm}}{\end{rm}\end{itproblem}}

\newcommand{\Xd}{(X, d)}

\newcommand{\M}{\mathcal{M}}

\newcommand{\Mzero}{\M_0}
\newcommand{\Mone}{\M_1}
\newcommand{\Moneplus}{\M_1^+}
\newcommand{\Mplus}{\M^+}

\newcommand{\MX}{\M(X)}
\newcommand{\MzeroX}{\M_0(X)}
\newcommand{\MoneX}{\M_1(X)}

\newcommand{\Ezero}{E_0}

\newcommand{\Imu}{I(\mu)}
\newcommand{\Inu}{I(\nu)}
\newcommand{\Imunu}{I(\mu, \nu)}

\newcommand{\N}{\mathbb{N}}
\newcommand{\R}{\mathbb{R}}

\newcommand{\supp}{\mathop{\rm supp}\nolimits}

\newcommand{\mbar}{M}

\newcommand{\ts}{\textstyle}

\allowdisplaybreaks

\newcommand{\nwrefA}[1]{\ref{#1} of~\cite{NW1}}

\begin{document}

\keywords{Compact metric space, finite metric space,
quasihypermetric space, space of negative type, metric embedding,
signed measure, signed measure of mass zero, spaces of measures,
distance geometry, geometric constant}

\subjclass[msc2000]{Primary 51K05; secondary 54E45, 31C45}

\title
{Distance Geometry in Quasihypermetric Spaces.~II}

\author{Peter Nickolas}
\address{School of Mathematics and Applied Statistics,
University of Wollongong, Wollongong, NSW 2522, Australia}
\email{peter\_\hspace{0.8pt}nickolas@uow.edu.au}
\author{Reinhard Wolf}
\address{Institut f\"ur Mathematik, Universit\"at Salz\-burg,
Hellbrunnerstrasse~34, A-5020 Salz\-burg, Austria}
\email{Reinhard.Wolf@sbg.ac.at}

\begin{abstract}
Let $\Xd$ be a compact metric space and let
$\MX$ denote the space of all finite signed Borel measures on~$X$.
Define $I \colon \MX \to \R$ by
\[
\Imu = \int_X \! \int_X d(x,y) \, d\mu(x) d\mu(y),
\]
and set
$
\mbar(X) = \sup \Imu,
$
where $\mu$ ranges over
the collection of signed measures in $\MX$ of total mass~$1$.
This paper, with an earlier and a subsequent paper
[Peter Nickolas and Reinhard Wolf,
\emph{Distance geometry in quasihypermetric spaces.\ I} and~\emph{III}],
investigates the geometric constant~$\mbar(X)$
and its relationship to the metric properties of~$X$ and the
functional-analytic properties of a certain subspace of $\MX$
when equipped with a natural semi-inner product.
Using the work of the earlier paper, this paper explores
measures which attain the supremum defining $\mbar(X)$,
sequences of measures which approximate the supremum when
the supremum is not attained and conditions implying or
equivalent to the finiteness of $\mbar(X)$.
\end{abstract}

\maketitle

\section{Introduction}
\label{Introduction2}

Let $\Xd$ be a compact metric space and let
$\MX$ denote the space of all finite signed Borel measures on~$X$.
Define functionals $I \colon \MX \times \MX \to \R$
and $I \colon \MX \to \R$ by
\[
I(\mu, \nu) = \int_X \! \int_X d(x,y) \, d\mu(x) d\nu(y)
\quad\mbox{and}\quad
\Imu = I(\mu, \mu) = \int_X \! \int_X d(x,y) \, d\mu(x) d\mu(y)
\]
for $\mu, \nu \in \MX$, and set
\[
\mbar(X) = \sup \Imu,
\]
where $\mu$ ranges over $\MoneX$,
the collection of signed measures in $\MX$ of total mass~$1$.

Our interest in this paper and in the earlier and later papers \cite{NW1}
and~\cite{NW3} is in the properties of the geometric constant $\mbar(X)$.
In~\cite{NW1}, we observed that if $\Xd$ does not have the quasihypermetric
property, then $\mbar(X)$ is infinite, and thus the context of
our study for the most part is that of quasihypermetric spaces.
Recall (see \cite{NW1}) that $\Xd$ is \textit{quasihypermetric} if
for all $n \in \N$, all $\alpha_1, \ldots, \alpha_n \in \R$ satisfying
$\sum_{i=1}^n \alpha_i = 0$, and all $x_1, \ldots, x_n \in X$, we have
\[
\sum_{i,j=1}^n \alpha_i \alpha_j d(x_i, x_j) \leq 0.
\]
(Other authors refer to quasihypermetric spaces, or their metrics,
as of \textit{negative type}; see, for example, \cite{Bret}
and~\cite{HKM}.)
It is straightforward to confirm that a compact metric space $\Xd$
is quasihypermetric if and only if $\Imu \leq 0$ for all
$\mu \in \MzeroX$, the subspace of $\MX$ consisting of
all signed measures of mass~$0$ (see Theorem~\nwrefA{qhmconds}).

In the presence of the quasihypermetric property,
a natural semi-inner product space structure becomes available
on $\MzeroX$. Specifically, for $\mu, \nu \in \MzeroX$, we define
\[
(\mu \mid \nu) = -\Imunu,
\]
and we denote the resulting semi-inner product space by $\Ezero(X)$.
The associated seminorm $\|\cdot\|$ on $\Ezero(X)$
is then given by
\[
\| \mu \| = \bigl[ -\Imu \bigr]^\frac{1}{2}.
\]

The semi-inner product space $\Ezero(X)$ is in many ways the key
to our analysis of the constant~$\mbar(X)$.
In~\cite{NW1}, we developed the properties of
$\Ezero(X)$ in a detailed way, exploring in particular
the properties of several operators and functionals associated
with $\Ezero(X)$, some questions related to its topology,
and the question of completeness. Questions directly relating
to the constant~$\mbar(X)$ were examined in~\cite{NW1} only
when they had a direct bearing on this general analysis.

In this paper, we use the framework provided by our work in~\cite{NW1}
to deal directly and in some detail with questions about~$\mbar(X)$.
Specifically, we discuss
\begin{enumerate}
\item[(1)]
maximal measures, that is, measures which attain the supremum defining $\mbar(X)$,
\item[(2)]
sequences of measures which approximate the supremum when
no maximal measure exists, and
\item[(3)]
conditions implying or equivalent to the finiteness of~$\mbar(X)$.
\end{enumerate}

We assume here that the reader has read~\cite{NW1}, and we repeat
its definitions and results here only as necessary. Also, in~\cite{NW1}
the background to our work, and in particular the contributions
of other authors
(see \cite{AandS, Bjo, FR1, Hin1, Wol1}, for example),
was discussed in some detail,
and this discussion will not be repeated here.
The paper~\cite{NW3} deals with further questions about~$\mbar(X)$,
relating especially to metric embeddings of~$X$ and the properties
of~$\mbar(X)$ when $X$ is a finite metric space.

\section{Definitions and Notation}
\label{chapter2a}

Let $\Xd$ (abbreviated when possible to~$X$) be a compact metric space. 
The diameter of~$X$ is denoted by~$D(X)$.
We denote by~$C(X)$ the Banach space of all real-valued
continuous functions on~$X$ equipped with the usual sup-norm.
Further,
\begin{itemize}
\item
$\MX$ denotes the space of all finite signed Borel measures on~$X$,
\item
$\Mzero(X)$ denotes the subspace of $\MX$
consisting of all measures of total mass~$0$,
\item
$\Mone(X)$ denotes the affine subspace of $\MX$
consisting of all measures of total mass~$1$,
\item
$\Mplus(X)$ denotes the set of all positive measures in~$\MX$, and
\item
$\Moneplus(X)$ denotes the intersection of $\Mplus(X)$ and $\Mone(X)$,
the set of all probability measures on~$X$.
\end{itemize}
For $x \in X$, the atomic measure at~$x$ is denoted by~$\delta_x$.

The functionals $I(\cdot, \cdot)$ and $I(\cdot)$ defined earlier
play a central role in our work.
A related functional $J(\cdot)$ on $\MX$ is
defined for each $\mu \in \MX$ by $J(\mu)(\nu) = \Imunu$
for $\nu \in \MX$.
For $\mu \in \MX$, the function $d_\mu \in C(X)$ is defined by
\[
d_\mu(x) = \int_X d(x, y) \, d\mu(y)
\]
for $x \in X$. Finally, as noted earlier, we define
\[
\mbar(X) = \sup \bigl\{ \Imu: \mu \in \Mone(X) \bigr\}.
\]

\section{Maximal and Invariant Measures}
\label{maxinvmeasures}

We call a measure $\mu \in \Mone(X)$ \textit{maximal} if $\Imu = \mbar(X)$,
and we call a measure $\mu \in \MX$ \textit{$d$-invariant}
if there exists $c \in \R$ such that $d_\mu(x) = c$ for all $x \in X$;
the number~$c$ is then called the \textit{value} of~$\mu$.

Our first result deals with the relationship between maximal
and invariant measures. (Recall from Definition~\nwrefA{sqhmdefn}
that a compact quasihypermetric space $\Xd$ is said to be strictly quasihypermetric
if $\Imu = 0$ only when $\mu = 0$, for $\mu \in \MzeroX$.)

\begin{theorem}
\label{2.2}
Let $\Xd$ be a compact metric space.
\begin{enumerate}
\item
If $\mu \in \Mone(X)$ is a maximal measure, then $\mu$ is $d$-invariant
with value $\mbar(X)$.
\item
If $X$ is quasihypermetric and if $\mu \in \Mone(X)$ is $d$-invariant
with value~$c$, then $\mu$ is maximal and $\mbar(X) = c$.
\item
If $X$ is strictly quasihypermetric, then there can exist
at most one maximal measure in~$\Mone(X)$.
\item
If $X$ is strictly quasihypermetric, then there can exist
at most one $d$-invariant measure in~$\MoneX$.
\end{enumerate}
\end{theorem}

\begin{proof}
(1)
We may clearly assume that $X$ is non-singleton, so that $\mbar(X) > 0$.
Let $\mu \in \Mone(X)$ be maximal.
Assume first that $d_\mu(x) > \mbar(X)$ for some $x \in X$.
Choose $\epsilon > 0$ such that $d_\mu(x) >\mbar(X) + \epsilon$,
and let $\alpha = \mbar(X) / (\mbar(X) + \epsilon)$.
Then for $\mu_{\alpha} \in \Mone(X)$ defined by
$\mu_{\alpha} = \alpha \mu + (1- \alpha) \delta_x$, we find
\begin{eqnarray*}
I (\mu_{\alpha})
 & = & \alpha^{2} I (\mu) + 2 \alpha (1-\alpha) \, d_\mu (x) \\*
 & = & \alpha^2 \mbar(X) + 2 \alpha (1-\alpha)\, d_\mu (x) \\
 & > & \alpha^2 \mbar(X) + 2 \alpha (1-\alpha)
       \bigl( \mbar(X) + \epsilon \bigr) \\
 & = & \mbar(X) \bigl( (\alpha - 1)^2 + 1 \bigr) \\*
 & > & \mbar(X),
\end{eqnarray*}
a contradiction.

Now assume that $d_\mu (x) < \mbar(X)$ for some $x \in X$.
Choose \mbox{$\epsilon > 0$} such that $d_\mu (x) < \mbar(X) - \epsilon$
and $\mbar(X) -2\epsilon > 0$,
and let $\alpha = \mbar(X) / (\mbar(X) - \epsilon)$.
Then for $\mu_{\alpha} \in \Mone(X)$ defined by
$\mu_{\alpha} = \alpha \mu + (1 - \alpha) \delta_x$, we find as before
that $I (\mu_{\alpha}) > \mbar(X)$, a contradiction.
It follows that $d_\mu(x) = \mbar(X)$ for all $x \in X$.

\smallskip
(2)
Let $d_\mu$ be a constant function on~$X$ with value $c \in \R$,
for some $\mu \in \Mone(X)$. For any $\nu \in \Mone(X)$,
we have $2 I(\mu, \nu) \geq I (\mu) + I (\nu)$
(see Theorem~\nwrefA{qhmconds}), and so
\begin{eqnarray*}
2c
 & = & 2c\nu (X) \\*
 & = & 2 \nu (d_\mu) \\
 & = & 2 I (\mu, \nu) \\
 & \geq & I (\mu) + I (\nu) \\
 & = & \mu (d_\mu) + I (\nu) \\
 & = & \mu (X) c + I (\nu) \\*
 & = & c + I (\nu).
\end{eqnarray*}
Therefore $\Inu \leq c$. Finally, $\Imu = c$ implies
$\mbar(X) = c$, so we conclude that $\mu$ is maximal.

\smallskip
(3)
Let $\mu$ and $\nu$ be two maximal
measures in $\Mone(X)$. Part~(1) implies that $d _\mu (x) = d_\nu (x)
= \mbar(X)$ for all $x \in X$. Therefore, if $\varphi = \mu - \nu$,
we have $\varphi \in \Mzero(X)$ and
$I(\varphi) = \varphi (d_{\varphi}) = \varphi(0) = 0$,
and hence $\varphi = 0$.

\smallskip
(4) This follows from (2) and~(3).
\end{proof}

Consider the strictly quasihypermetric space $X = [a,b]$,
with the usual metric. Theorem~\ref{2.2} gives
a completely elementary proof that $\mbar([a,b]) < \infty$
(compare \cite[Lemma~3.5]{AandS}).

\begin{corollary}
\label{2.3}
Let $X = [a, b]$, where $a, b \in \R$ and $a < b$, and let
$d(x, y) = |x - y|$ for all $x, y \in [a, b]$.
Then $\mbar([a, b]) = (b-a)/2$.
\end{corollary}

\begin{proof}
Let $\mu \in \Mone([a, b])$
be defined by $\mu = \frac{1}{2} (\delta_a + \delta_b)$. Clearly,
we have $d_\mu (x) = (b-a)/2$ for all $x \in [a, b]$. Therefore,
by Theorem~\ref{2.2} part~(2), we have
$\mbar([a, b]) = (b-a)/2$.
\end{proof}

Furthermore, we can apply Theorem~\ref{2.2} to the quasihypermetric
but not strictly quasihypermetric space $X = S^1$,
the circle of radius~$1$, equipped with the arc-length metric
(see Example~\nwrefA{nonstrict}). Indeed, an identical argument
and conclusion apply to the sphere $S^{n-1}$ in~$\R^n$
equipped with the great-circle metric.

\begin{corollary}
\label{2.4}
Let $X = S^1$, the circle of radius~$1$,
equipped with the arc-length metric.
Then we have $\mbar(X) = \frac{\pi}{2}$.
Moreover, $X$ has multiple maximal/$d$-invariant measures.
\end{corollary}

\begin{proof}
Let $x_1$ and $y_1$ be any pair of diametrically
opposite points in~$X$ and let $\mu \in \Mone(X)$ be defined by
$\mu = \frac{1}{2}(\delta_{x_1} + \delta_{y_1})$. Clearly,
we have $d_\mu(x) = \frac{\pi}{2}$ for all $x \in X$, and hence
$\mbar(X) = \frac{\pi}{2}$, by Theorem~\ref{2.2} part~(2).
Since $x_1$ and $y_1$ can be chosen arbitrarily, the second claim holds.
\end{proof}

\begin{example}
\label{2.5}
Consider the compact strictly quasihypermetric space $X = B^3$,
the closed ball of radius~$1$ in~$\R^3$,
with the usual euclidean metric.
It is shown in~\cite{Alex1}
that $\mbar(X) = 2$ and that $\Imu < 2$ for all $\mu \in \Mone(X)$,
and that there therefore exists no maximal or invariant measure on~$X$.
\end{example}

The next result will provide us with a fruitful source
of examples and counterexamples in our later work.

\begin{theorem}
\label{2.1}
Let $(X, d_1)$ and $(Y, d_2)$
be compact metric spaces with $X \cap Y = \emptyset$
and $\mbar(X), \mbar(Y) < \infty$.
Let $Z = X \cup Y$, and define $d \colon Z \times Z \to \R$ by setting
\[
d(x, y)
 =
\begin{cases}
  \begin{array}{ll}
    d_1 (x, y), & \mbox{for } x, y \in X, \\
    d_2(x, y), & \mbox{for } x, y \in Y, \\
    c, & \mbox{for } x \in X, y \in Y,
  \end{array}
\end{cases}
\]
where $c \in \R$ is such that $2c \geq \max (D(X), D(Y))$.
Then we have the following.
\begin{enumerate}
\item
$(Z, d)$ is a compact metric space.
\item
If $X$ and~$Y$ are quasihypermetric,
then $(Z, d)$ is quasihypermetric if and only if
$2c \geq \mbar(X) + \mbar(Y)$.
\item
If $X$ and~$Y$ are strictly quasihypermetric,
then $(Z, d)$ is strictly quasihypermetric if and only if
$2c \geq \mbar(X) + \mbar(Y)$ and
\begin{itemize}
\item[(a)] $2c > \mbar(X) + \mbar(Y)$ or
\item[(b)] $X$ has no maximal measure or
\item[(c)] $Y$ has no maximal measure.
\end{itemize}
\end{enumerate}
\end{theorem}

\begin{proof}
It is straightforward to check that $(Z, d)$ is a compact metric space.

Consider $\mu \in \Mzero(Z)$.
Then we have $\mu = \mu_1 + \mu_2$, with $\supp(\mu_1)\subseteq X$ and
$\supp(\mu_2) \subseteq Y$, so we can regard $\mu_1, \mu_2$ as
members of $\MX, \M(Y)$, respectively, and since $\mu \in \Mzero (Z)$
we have $0 = \mu_1 (Z) + \mu_2 (Z) = \mu_1(X) + \mu_2(Y)$.
Note that
$\Imu = I(\mu_1 + \mu_2) = I(\mu_1) + I(\mu_2) + 2c\mu_1(X) \mu_2(Y)$.

Suppose that $\mu_1(X) = 0$, so that $\mu_2(Y) = 0$ also.
If $X$ and~$Y$ are quasihypermetric, we therefore have
$\Imu = I(\mu_1) + I(\mu_2) \leq 0$,
and if $X$ and~$Y$ are moreover strictly quasihypermetric,
then $\Imu = 0$ implies $I(\mu_1) = I(\mu_2) = 0$,
and hence we have
$\mu_1 = \mu_2 = 0$, and so $\mu = 0$.

Now suppose that $\mu_1(X) \neq 0$, so that $\mu_2(Y) = -\mu_1(X) \neq 0$.
Then we find
\begin{eqnarray*}
\Imu
 & = & \mu^2_1(X) I \Bigl(\frac{\mu_1}{\mu_1(X)}\Bigr)
     + \mu^2_2(Y) I \Bigl(\frac{\mu_2}{\mu_2(Y)}\Bigr)
     + 2c \mu_1 (X) \mu_2 (Y) \\
 & = & \mu^2_1(X) \left[I \Bigl(\frac{\mu_1}{\mu_1(X)}\Bigr)
     + I \Bigl(\frac{\mu_2}{\mu_2(Y)}\Bigr)-2c \right] \\
 & \leq & \mu^2_1(X) \left[\mbar(X) + \mbar(Y) - 2c \right]. 
\end{eqnarray*}
If $2c \geq \mbar(X) + \mbar(Y)$, it follows immediately
that $\Imu \leq 0$, and also that $\Imu < 0$ if $\mbar(X) + \mbar(Y) < 2c$
or $I(\mu_1/\mu_1 (X)) < \mbar(X)$ or
$I(\mu_2/\mu_2(Y))< \mbar(Y)$.
This proves the reverse implications in (2) and~(3).

For the forward implication in~(2), suppose that $X$ and~$Y$ are quasihypermetric
and that $2c < \mbar(X) + \mbar(Y)$.
Then there exist $\mu_1 \in \MoneX$ and $\mu_2 \in \Mone(Y)$ such that
$2c < I(\mu_1) + I(\mu_2)$. But now $\mu = \mu_1 -\mu_2 \in \Mzero(Z)$, and
\[
\Imu = I(\mu_1) + I(\mu_2) - 2I(\mu_1, \mu_2) > 2c - 2I(\mu_1, \mu_2) = 0,
\]
and hence $Z$ is not quasihypermetric.

For the forward implication in~(3),
suppose that $X$ and~$Y$ are strictly quasihypermetric.
By part~(2), if $2c < \mbar(X) + \mbar(Y)$,
then $Z$ is not (strictly) quasihypermetric,
so let us assume that $2c \geq \mbar(X) + \mbar(Y)$ and that conditions
(a), (b) and~(c) in~(3) are false.
Thus we have $2c = \mbar(X) + \mbar(Y)$ and there exist
maximal measures $\mu_1 \in \MoneX$ and $\mu_2 \in \Mone(Y)$.
But now $\mu = \mu_1 - \mu_2 \in \Mzero(Z)$ is non-zero, and
\[
\Imu
 = I(\mu_1) + I(\mu_2) - 2I(\mu_1, \mu_2)
 = \mbar(X) + \mbar(Y) - 2c
 = 0,
\]
and hence $Z$ is not strictly quasihypermetric.
\end{proof}

\begin{theorem}
\label{2.1.4}
Let the metric spaces $(X, d_1)$, $(Y, d_2)$ and $(Z, d)$,
and the constant~$c$ satisfying $2c \geq \max (D(X), D(Y))$,
be as in Theorem~\ref{2.1}, with $X$ and~$Y$ quasihypermetric.
Suppose that $\mu_1 \in \MoneX$, $\mu_2 \in \Mone(Y)$
are invariant measures.
Then
\[
\mu = (\mbar(Y) - c) \mu_1 + (\mbar(X) - c) \mu_2
\]
is an invariant measure on~$Z$ with value $\mbar(X) \mbar(Y) -c^2$.
Further, if $X$ and~$Y$ are strictly quasihypermetric,
then $\mu \in \Mzero(Z)$ if and only if $Z$~is quasihypermetric
but not strictly quasihypermetric.
\end{theorem}

\begin{proof}
The first statement is proved by a straightforward calculation.
For the second, we note that by Theorem~\ref{2.1}
and Theorem~\ref{2.2} part~(2), $Z$~is quasihypermetric
but not strictly quasihypermetric
if and only if $2c = \mbar(X) + \mbar(Y)$,
from which the statement follows.
\end{proof}

\begin{remark}
\label{nonqhmcase}
In the presence of the quasihypermetric property,
Theorem~\ref{2.2} above shows that invariant measures
are maximal, and conversely.
When $X$ is not quasihypermetric, on the other hand,
Theorem~\nwrefA{2.6} (see Theorem~\ref{2.6new} below)
shows that $\mbar(X) = \infty$, and that the notion
of a maximal measure is therefore meaningless.
In the following result, we show that a non-quasihypermetric
space may nevertheless have a non-trivial invariant measure.
\end{remark}

\begin{theorem}
\label{2.9A}
There exists a $5$-point non-quasihypermetric space
with an invariant probability measure of value~$1$.
\end{theorem}

\begin{proof}
For $X = \{ x_1, x_2 \}$ and $Y = \{y_1, y_2, y_3\}$, define
\[
d_1(x_i, x_j)
 =
\begin{cases}
  0 & i = j, \\
  2 & i \neq j,
\end{cases}
\qquad
d_2 (y_i, y_j)
 =
\begin{cases}
  0 & i = j, \\
  2 & i \neq j.
\end{cases}
\]
Then $(X, d_1)$ and $(Y, d_2)$ are compact quasihypermetric spaces and
$\mu_1 = \frac{1}{2}(\delta_{x_1} + \delta_{x_2}) \in \Mone(X)$
and
$\mu_2 = \frac{1}{3}(\delta_{y_1} + \delta_{y_2} + \delta_{y_3}) \in \Mone(Y)$
are invariant measures. Theorem~\ref{2.2} part~(2) then
implies that they are maximal measures, so that we have
$\mbar(X) = 1$ and $\mbar(Y) = \frac{4}{3}$.

Now let $Z = X \cup Y$, and define $d \colon Z \times Z \to \R$
as in Theorem~\ref{2.1}, with $c = 1$.
Then parts (1) and~(2) of Theorem~\ref{2.1} imply that $(Z, d)$
is non-quasihypermetric, and we find that $\mu_1 \in \Moneplus(Z)$
satisfies $d_{\mu_1}(z) = 1$ for all $z \in Z$.
\end{proof}

\section{Maximal and Invariant Sequences}
\label{maxinvsequences}

\begin{definition}
\label{definitionA}
Let ($X, d$) be a compact quasihypermetric space with
$\mbar(X) < \infty$. A sequence $\mu_n$ in $\Mone(X)$ is called
\textit{maximal} if $I(\mu_n) \to \mbar(X)$ as $n \to \infty$.
\end{definition}

\begin{remark}
\label{remarkB1}
While Example~\ref{2.5} shows that maximal measures may not exist
under the assumption that $\mbar(X) < \infty$, it is of course
immediate from the definition that maximal sequences always exist.
\end{remark}

We noted in section~\ref{topologies} of~\cite{NW1} that
when $\mbar(X) < \infty$ there is a natural extension of the semi-inner
product on $\Ezero(X) = \Mzero(X)$ to a semi-inner product
on the space $\M(X)$ of all signed Borel measures on~$X$.
Specifically, we define
\[
(\mu \mid \nu) = (\mbar(X) + 1) \mu(X) \nu(X) -\Imunu
\]
for $\mu, \nu \in \MX$, and we denote the resulting semi-inner product
space by $E(X)$. This space plays a role in the arguments below.

\begin{remark}
\label{remarkB2}
In the context of the semi-inner product space $E(X)$,
maximal sequences and maximal measures have the following
natural interpretation.
\begin{enumerate}
\item[(1)]
A sequence $\mu_n$ in $\Mone(X)$ is maximal if and only if
$\| \mu_n \| \to \mbox{dist}(0, \Mone(X)) = 1$
as $n \to \infty$, where $\mbox{dist}(0, \Mone(X))$
denotes the distance of the zero measure to the closed affine
subspace $\Mone(X)$ (see Corollary~\nwrefA{2.10.6}).
\item[(2)]
A measure $\mu \in \Mone(X)$ is maximal if and
only if
\[
\| \mu \| = \mbox{dist}(0, \Mone(X)) = 1.
\]
\end{enumerate}
The preceding assertions follow immediately from the observation that
$\| \mu \|^2 = \mbar(X) + 1 - I(\mu )$
for all $\mu \in \Mone(X)$.
\end{remark}

\begin{remark}
\label{remarkB3}
A measure $\mu$ in $\Mone(X)$ is maximal if and only
if there exists a maximal sequence $\mu_n$ in $\Mone(X)$
such that $\| \mu_n - \mu \| \to 0$ as $n \to \infty$.
For if $\mu_n$ is such a sequence,
then $\| \mu \| \leq \| \mu_n \| + \| \mu_n - \mu \|$
for all $n \in \N $, so $\| \mu \| \leq 1$
by Remark~\ref{remarkB2} part~(1), and the maximality of~$\mu$
follows by Remark~\ref{remarkB2} part~(2).
\end{remark}

Recall that if $X$ is a compact quasihypermetric space, then maximal
measures in $\Mone(X)$, if they exist, are characterized by
the property that they are $d$-invariant on~$X$ (see Theorem~\ref{2.2}).
However, there exist spaces~$X$ with $\mbar(X) < \infty$
but without maximal measures (see Example~\ref{2.5}).
In the light of these facts, we make the following definition.

\begin{definition}
\label{2.24}
Let $\Xd$ be a compact quasihypermetric space. A sequence $\mu_n$
in $\Mone(X)$ is called \textit{$d$-invariant with value~$c$},
for some $c \in \R$, if
\begin{enumerate}
\item[(1)]
$\| \mu_n - \mu_m \| \to 0$ as $n, m \to \infty$, and
\item[(2)]
$d_{\mu_n} \to c \cdot \underline{1}$ in $C(X)$ as $n \to \infty$,
where $\underline{1} \in C(X)$ is defined by $\underline{1}(x): = 1$
for all $x \in X$.
\end{enumerate}
\end{definition}

We wish now to investigate the relationship between maximal
and invariant sequences (cf.~Theorem~\ref{2.2} above).
We need first the following three lemmas.

\begin{lemma}
\label{lemmaC}
Let $\Xd$ be a compact quasihypermetric space with $\mbar(X) < \infty$.
A sequence $\mu_n$ in $\Mone(X)$ is maximal if and only if
$\| \mu_n - \mu_m \| \to 0$ as $n, m \to \infty$ and
$( \mu_n \mid \nu) \to 0$ as $n \to \infty$ for all $\nu \in \Ezero(X)$.
\end{lemma}

\begin{proof}
The assertion is a well-known fact about semi-inner product spaces,
but for completeness we include a proof.
Let $\mu_n$ be a maximal sequence in $\Mone(X)$. Since 
$\| \mu \|^2 = \mbar(X) + 1 - I(\mu ) \geq 1$ for all $\mu \in \Mone(X)$
(see Remark~\ref{remarkB2}), we have
\begin{eqnarray*}
\frac{\| \mu_n - \mu_m \|^2}{4}
 & = &
\frac{\| \mu_n\|^2 + \| \mu_m \|^2}{2}
  - \Big \| \frac{\mu_n + \mu_m}{2} \Big \|^2 \\
 & \leq &
\frac{\| \mu_n \| ^2 + \| \mu_m \| ^2}{2} - 1, 
\end{eqnarray*}
for all $n, m$, and we conclude by Remark~\ref{remarkB2} part~(1) that
$\| \mu_n - \mu_m \| \to 0$ as $n, m \to \infty$.

Let $\nu \in \Ezero(X)$. From the fact that $\| \mu_n - \mu_m \| \to 0$
as $n, m \to \infty$ it follows that $(\mu_n \mid \nu ) \to \alpha$
as $n \to \infty$, for some $\alpha \in \R$.
Then since $1 \leq \| \mu_n + t \nu \| ^2$ for all $t \in \R$,
it follows that $0 \leq 2 t \alpha + t^2 \| \nu \| ^2$
for all $t \in \R$, and hence $\alpha = 0$.

Conversely, let $\mu_n$ be a sequence in $\Mone(X)$
such that $\| \mu_n - \mu_m \| \to 0$ as $n, m \to \infty$ and
$( \mu_n \mid \nu ) \to 0$ as $n \to \infty$ for all $\nu \in \Ezero(X)$.
Since the measures~$\mu_n$ form a Cauchy sequence in~$\MX$,
the norms~$\|\mu_n\|$ form a convergent sequence in~$\R$,
and we define $\beta := \lim_{n\to\infty} \| \mu_n \|$.
Fix $\mu \in \Mone(X)$ and $\epsilon > 0$.
Now choose $N \in \N$ such that
$\| \mu_N \| ^2 \geq \beta ^2 - \frac{\epsilon}{2}$ and
$(\| \mu \| + \| \mu_N \|) \| \mu_N - \mu_m \| < \frac{\epsilon}{4}$
for all $m \geq N$. Then
\begin{eqnarray*}
\| \mu \|^2
 & = &
\| \mu - \mu_N \| ^2 + \| \mu_N \| ^2
    + 2 (\mu - \mu_N \mid \mu_N ) \\
 & = &
\| \mu - \mu_N \| ^2 + \| \mu_N \| ^2
   + 2 (\mu - \mu_N \mid \mu_N - \mu_m) + 2 (\mu - \mu_N \mid \mu_m) \\
 & \geq &
\| \mu_N \| ^2 - 2 (\| \mu \| + \|\mu_N \|)  \| \mu_N - \mu_m \|
  + 2 (\mu - \mu_N \mid \mu_m) \\
 & \geq &
\beta ^2 - \epsilon + 2(\mu - \mu_N \mid \mu_m)
\end{eqnarray*}
for all $m \geq N$, and, using the fact that $\mu - \mu_N \in \Ezero(X)$,
we let $m \to \infty$, obtaining $\| \mu \| ^2 \geq \beta ^2 - \epsilon$.
But $\mu$ and~$\epsilon$ were arbitrary, so it follows that
$\| \mu \| \geq \lim_n \| \mu_n \|$ for all $\mu \in \Mone(X)$.
Therefore, $\| \mu_n \| \to \mbox{dist}(0, \Mone(X))$,
and so, by Remark~\ref{remarkB2} part~(1), we are done.
\end{proof}

\begin{lemma}
\label{2.25}
Let $\Xd$ be a compact quasihypermetric space.
If there exist a sequence $\mu_n$ in $\Mone(X)$ and constants
$\alpha, \beta \in \R$ such that $I(\mu_n) \to \alpha$ and
$d_{\mu_n} \to \beta \cdot \underline{1}$ in $C(X)$ as $n \to \infty$,
then $\mbar(X) \leq 2 \beta - \alpha < \infty$.
\end{lemma}

\begin{proof}
Let $\mu$ be in $\Mone(X)$.
Since $d_{\mu_n} \to \beta \cdot \underline{1}$ in $C(X)$,
we have $I(\mu, \mu_n) \to \beta$.
But it is an easy consequence of the quasihypermetric property
(see part~(5) of Theorem~\nwrefA{qhmconds}) that 
$2 I (\mu, \mu_n) \geq \Imu + I(\mu_n)$ for all $n \in \N$,
which implies that $\Imu \leq 2 \beta - \alpha$,
and hence we have $\mbar(X) \leq 2 \beta - \alpha < \infty$.
\end{proof}

\begin{lemma}
\label{2.26}
Let $\Xd$ be a compact quasihypermetric space
with $\mbar(X) < \infty$. If $\mu_n$ in $\Mone(X)$
is a $d$-invariant sequence with value~$c$,
then $I(\mu_n) \to c$ as $n \to \infty$.
\end{lemma}

\begin{proof}
Let $\epsilon > 0$. By assumption, there exists $N \in \N$
such that $\| \mu_n - \mu_m \| < \epsilon$ for all $n, m \geq N$,
and there exists $K > 0$ such that $\| \mu_n \| \leq K$
for all $n \in \N$. Therefore, for all $n > N$ we have
\begin{eqnarray*}
\bigl| I (\mu_n) - c \bigr|
 & \leq &
\bigl| I (\mu_n, \mu_n - \mu_N) \bigr| + \bigl| I(\mu_n, \mu_N) - c \bigr| \\
 & = &
\bigl| (\mu_n \mid \mu_n - \mu_N) \bigr|
    + \bigl| \mu_N (d_{\mu_n}) - c \bigr| \\
 & \leq &
\| \mu_n \| \cdot \| \mu_n - \mu_N \|
    + \bigl| \mu_N (d_{\mu_n}) - c \bigr| \\
 & \leq &
\epsilon \cdot K + \bigl| \mu_N (d_{\mu_n}) - c \bigr|.
\end{eqnarray*}
But since $d_{\mu_n} \to c \cdot \underline{1}$ in $C(X)$,
we have $\mu_N (d_{\mu_n}) \to c$ as $n \to \infty$, and the result follows.
\end{proof}

Now we can prove the following counterpart of Theorem~\ref{2.2}
for sequences of measures.

\begin{theorem}
\label{theoremD}
Let $\Xd$ be a compact quasihypermetric space.
Then we have the following.
\begin{enumerate}
\item[(1)]
If $\mbar(X) < \infty$ and $\mu_n$ is a maximal sequence
in $\Mone(X)$, then $\mu_n$~is a $d$-invariant sequence
with value $\mbar(X)$.
\item[(2)]
If $\mu_n$ is a $d$-invariant sequence in $\Mone(X)$ with value~$c$,
then $\mbar(X) = c < \infty$ and $\mu_n$ is a maximal sequence.
\end{enumerate}
\end{theorem}

\begin{proof}
(1)
Let $\mu_n$ be a maximal sequence in $\Mone(X)$.
Since $\| \mu_n - \mu_m \| \to 0$ as $n, m \to \infty$,
by Lemma~\ref{lemmaC}, it follows by part (2) of Theorem~\nwrefA{2.10.5}
that $\| d_{\mu_n} - d_{\mu_m} \| \to 0$ as $n, m \to \infty$.
Since $C(X)$ is complete, there exists $f \in C(X)$
such that $d_{\mu_n} \to f \in C(X)$ as $n \to \infty$,
so that $d_{\mu_n} (x) \to f(x)$ as $n \to \infty$ for all $x \in X$.

Lemma~\ref{lemmaC} again tells us that $(\mu_n \mid \nu ) \to 0$
as $n \to \infty$, for all $\nu \in \Ezero(X)$.
In particular, $(\mu_n \mid \delta_x - \delta_y) \to 0$
as $n \to \infty$, for all $x, y \in X$.
Thus we have both $d_{\mu_n} (x) - d_{\mu_n} (y) \to 0$
and $d_{\mu_n} (x) \to f(x)$ as $n \to \infty$, for all $x, y \in X$,
and we conclude that $f(x) = f(y)$ for all $x, y \in X$.
Thus $\mu_n$ is a $d$-invariant sequence, and by Lemma~\ref{2.26}
its value is $\mbar(X)$.

\smallskip
(2)
Let $\mu_n$ be a $d$-invariant sequence in $\Mone(X)$ with
value~$c$, and fix $x \in X$. It follows immediately from
the definition of $d$-invariance that $\mu_n - \delta_x$
is a Cauchy sequence in $\Ezero(X)$, and so there exists
$\alpha \geq 0$ such that $\| \mu_n - \delta_x \| \to \alpha$
as $n \to \infty$. This gives
$2 I (\mu_n, \delta_x) - I (\mu_n) \to \alpha ^2$ as $n \to \infty$,
and since $d_{\mu_n}(x) \to c$ as $n \to \infty$,
we have $I (\mu_n) \to 2c - \alpha ^2$. Applying Lemma~\ref{2.25},
we find that $\mbar(X) \leq 2c - (2c - \alpha ^2) = \alpha ^2 < \infty$.
Thus Lemma~\ref{2.26} applies, showing that
$I(\mu_n) \to c$ as $n \to \infty$,
and it follows that $c = \alpha ^2$. Therefore $\mbar(X) \leq c$,
and since $I(\mu_n) \to c$ as $n \to \infty$,
we have $\mbar(X) = c$ and $I(\mu_n) \to \mbar(X)$,
and so $\mu_n$ is a maximal sequence.
\end{proof}

The equivalence of parts (1) and~(2) in the following result
is merely a restatement of the definition of a maximal
sequence, while the equivalence of parts (2) and~(3) is
essentially a restatement of Theorem~\ref{theoremD}.

\begin{corollary}
\label{theoremDnew}
Let $\Xd$ be a compact quasihypermetric space. Then the following
conditions are equivalent.
\begin{enumerate}
\item[(1)]
$\mbar(X) < \infty$.
\item[(2)]
There exists a maximal sequence in $\Mone(X)$.
\item[(3)]
There exists a $d$-invariant sequence in $\Mone(X)$.
\end{enumerate}
\end{corollary}

This result takes on an especially pleasant form
in the case of a finite space (see also Theorem~3.4 of~\cite{NW3}).

\begin{theorem}
\label{theoremDcorollary}
Let $\Xd$ be a finite quasihypermetric space. Then the following
conditions are equivalent.
\begin{enumerate}
\item[(1)]
$\mbar(X) < \infty$.
\item[(2)]
There exists a maximal measure in $\Mone(X)$.
\item[(3)]
There exists a $d$-invariant measure in $\Mone(X)$.
\end{enumerate}
\end{theorem}

\begin{proof}
The equivalence of (2) and~(3) is given by Theorem~\ref{2.2},
and the fact that (2) implies~(1) is trivial, so we need only
confirm that (1) implies~(2).
If $\mbar(X) < \infty$, then Theorem~\ref{theoremD}
tells us that there exists a sequence $\mu_n$ in $\Mone(X)$
which is maximal and is $d$-invariant with value $\mbar(X)$.
By the definition of $d$-invariance, the sequence~$\mu_n$
is a Cauchy sequence in the semi-inner product space $E(X)$,
which, since $X$ is finite, is complete, by Theorem~\nwrefA{2.31}.
Choose $\mu \in \MX$ such that \mbox{$\mu_n \to \mu$} as $n \to \infty$.
By Corollary~\nwrefA{2.10.6}, the subspace $\Mzero(X) = \Ezero(X)$
is closed in $E(X)$, and therefore so is its translate $\Mone(X)$,
and it follows that $\mu \in \Mone(X)$. Now, by Remark~\ref{remarkB3},
we conclude that $\mu$ is a maximal measure, and so (1) implies~(2),
as required.
\end{proof}

\section{The Finiteness of $\mbar(X)$}
\label{closerdiscussion}

We now turn to discussion of $\mbar(X)$, focusing especially on
the circumstances under which $\mbar(X)$ is finite. We begin by
recalling two of our earlier results from~\cite{NW1}
which give information on this question.

\begin{theorem}[= Theorem~\nwrefA{2.6}]
\label{2.6new}
If $\Xd$ is a compact non-quasihypermetric space, then $\mbar(X) = \infty$.
\end{theorem}

\begin{theorem}[= Theorem~\nwrefA{2.9.5}]
\label{2.9.5new}
Let $\Xd$ be a compact quasihypermetric space.
If there exists $\mu \in \Mzero(X)$ which is $d$-invariant
with value $c \neq 0$, then
\begin{enumerate}
\item[(1)]
$X$ is not strictly quasihypermetric and
\item[(2)]
$\mbar(X) = \infty$.
\end{enumerate}
\end{theorem}

If $X$ is a finite space, we can give more information.

\begin{theorem}
\label{2.13}
Let $\Xd$ be a finite quasihypermetric space. Then we have the following.
\begin{enumerate}
\item[(1)]
If $X$ is strictly quasihypermetric, then $\mbar(X) < \infty$.
\item[(2)]
If $X$ is not strictly quasihypermetric,
then $\mbar(X) < \infty$ if and only if there exists no
$d$-invariant measure $\mu \in \Mzero(X)$ with value $c \neq 0$.
\end{enumerate}
\end{theorem}

\begin{proof}
(1)
Since $X$ is finite, it follows that $\Ezero(X)$
is a finite-dimensional normed space, and hence $J(\mu)$
(see section~\ref{chapter2a})
is a bounded linear functional on $\Ezero(X)$ for each $\mu \in \MX$.
Therefore, part~(3) of Theorem~\nwrefA{2.10}
(see also Remark~\nwrefA{2.11}) implies the assertion.

\smallskip
(2)
Theorem~\ref{2.9.5new} part~(2) deals immediately
with the forward implication. For the reverse implication,
assume that no measure $\mu \in \Mzero(X)$
and constant $c \neq 0$ exist with the property that
$d_\mu(x) = c$ for all $x \in X$.
Fix $x \in X$, and define $f \colon \Ezero(X)/F \to \R$ by setting
$f(\nu + F) = I(\delta_x, \nu)$ for $\nu + F \in \Ezero(X)/F$.
(Recall from Lemma~\nwrefA{2.7} that $F$ denotes the subspace
$\{ \mu \in \Ezero(X) : \| \mu \| = 0 \}$ of~$\Ezero(X)$.)

If $\nu + F = \nu' + F$, we have $\nu-\nu' \in F$, and hence,
by part~(5) of Lemma~\nwrefA{2.7}, there exists $\gamma \in \R$
such that $d_{\nu-\nu'}(x) = \gamma$ for all $x \in X$.
By our assumption, we have $\gamma = 0$, and hence
$I(\delta_x, \nu) = d_\nu(x) = d_{\nu'}(x) = I(\delta_x, \nu')$.
Thus $f$ is a well defined linear functional on the finite-dimensional
normed space $\Ezero(X)/F$ (see part~(4) of Lemma~\nwrefA{2.7}).
Therefore, $f$ is bounded on $\Ezero(X)/F$,
and so there exists $M \geq 0$ such that
\[
\bigl| I(\delta_x, \nu) \bigr|
 = \bigl| f(\nu + F) \bigr|
 \leq M \, \| \nu + F \|
 = M \, \| \nu \|
\]
for all $\nu \in \Ezero(X)$. Now part~(2) of Theorem~\nwrefA{2.10}
implies that $\mbar(X) < \infty$, as required.
\end{proof}

\begin{theorem}
\label{2.9}
There exists a $5$-point quasihypermetric, non-strictly quasihypermetric space~$Z$ with
$\mbar(Z) = \infty$.
\end{theorem}

\begin{proof}
For $X = \{ x_1, x_2 \}$ and $Y = \{y_1, y_2, y_3\}$, define
\[
d_1(x_i, x_j)
 =
\begin{cases}
  0 & i = j, \\
  1 & i \neq j,
\end{cases}
\qquad
d_2 (y_i, y_j)
 =
\begin{cases}
  0 & i = j, \\
  \frac{4}{5} & i \neq j.
\end{cases}
\]
It is easy to check that $(X, d_1)$ and $(Y, d_2)$ are
compact strictly quasihypermetric spaces. It is clear that
$\mu_1 = \frac{1}{2}(\delta_{x_1} + \delta_{x_2}) \in \Mone(X)$
and
$\mu_2 = \frac{1}{3}(\delta_{y_1} + \delta_{y_2} + \delta_{y_3}) \in \Mone(Y)$
are invariant measures, and Theorem~\ref{2.2} part~(2) then
implies that they are maximal measures, so that we have
$\mbar(X) = \frac{1}{2}$ and $\mbar(Y) = \frac{8}{15}$.

If we let $Z = X \cup Y$ and define $d \colon Z \times Z \to \R$
as in Theorem~\ref{2.1}, with 
$c = \frac{1}{2} \bigl( \mbar(X) + \mbar(Y) \bigr) = \frac{31}{60}$,
then Theorem~\ref{2.1} shows that $(Z, d)$ is a quasihypermetric space.
Further, Theorem~\ref{2.1.4} implies that
$\mu = \mu_1 - \mu_2 \in \Mzero(X)$ is $d$-invariant
with value $-\frac{1}{60} \neq 0$,
and then finally Theorem~\ref{2.9.5new} implies that $Z$ is
not strictly quasihypermetric and that $\mbar(Z) = \infty$.
\end{proof}

\begin{remark}
\label{5not4}
Theorem~\ref{2.9} constructs a space with $5$ points.
We note that $5$ is the smallest number possible in such an example:
in Theorem~5.6 of~\cite{NW3}, we show among other things
that every metric space with $4$ or fewer points must have
$\mbar(X) < \infty$.
\end{remark}

To complete our survey of the finiteness or otherwise of~$\mbar$,
we require the following result and example.

\begin{theorem}
\label{sqhmmbarinf}
There exists a compact strictly quasihypermetric space~$Z$ with
\mbox{$\mbar(Z) = \infty$}.
\end{theorem}

\begin{proof}
Choose a strictly quasihypermetric compact space $(X, d_1)$
without a maximal measure and with $\mbar(X) < \infty$
(see Example~\ref{2.5}). Also, let $Y = \{y_1, y_2\}$, and define
\[
d_2(y_i, y_j)
 =
\begin{cases}
  0, & i = j, \\
  D(X), & i \neq j,
\end{cases}
\qquad
 1 \leq i, j \leq 2.
\]
Of course, $(Y, d_2)$ is a compact strictly quasihypermetric space
with $\mbar(Y) = D(X)/2$.

Let $Z = X \cup Y$, and define $d \colon Z \times Z \to \R$
as in Theorem~\ref{2.1}, with $c = \frac{1}{2}(\mbar(X) + \mbar(Y))$.
Choose two points $x_1, x_2 \in X$ with $D(X) = d(x_1, x_2)$,
and note that
$I \bigl( \frac{1}{2}(\delta_{x_1} + \delta_{x_2}) \bigr) = D(X)/2$.

Since $\mbar(X)$ is not attained, we have
$\mbar(X)
 > I \bigl( \frac{1}{2}(\delta_{x_1} + \delta_{x_2}) \bigr)
 = D(X)/2$.
Hence $2c \geq \max(\mbar(X) + \mbar(Y), D(X), D(Y))$,
and Theorem~\ref{2.1} implies that $(Z, d)$ is a compact strictly quasihypermetric space.

Now choose $\mu_n \in \Mone(X)$ for each~$n$ such that
$I(\mu_n) \to \mbar(X) < \infty$ as $n \to \infty$
(we assume that $\mbar(X) < \infty$,
since there is otherwise nothing to prove).
Define $\nu_n \in \Mone(Z)$ by setting
\[
\nu_n = \alpha_n \mu_n
 + {\ts\frac{1}{2}} (1 - \alpha_n) (\delta_{y_1} + \delta_{y_2}),
\]
where $\alpha_n = (\mbar(X)-I(\mu_n))^{- \frac{1}{2}}$.
By assumption, $\alpha_n$ is well defined, $\alpha_n > 0$,
and $\alpha_n \to \infty$ as $n \to \infty$.
Now
\[
I(\nu_n) = \alpha^{2}_{n} I(\mu_n) + 2c \alpha_n
(1 - \alpha_n) + {\ts\frac{1}{2}} (1 - \alpha_n)^2 D(X),
\]
and, expanding and simplifying, we find finally that
\[
I(\nu_n)
 = -1 + {\ts\frac{1}{2}} D(X)
  + \alpha_n \left(\mbar(X) - {\ts\frac{1}{2}} D(X) \right)
 \to \infty
\]
as $n \to \infty$, since $\mbar(X) - {\ts\frac{1}{2}} D(X) > 0$,
giving the result.
\end{proof}

\begin{example}
\label{4ptexample}
Let $X$ be a $4$-point space consisting of any two pairs
of diametrically opposite points chosen from the circle of radius~$1$
with the arc-length metric.
By Example~\nwrefA{nonstrict}, $X$ is quasihypermetric
but not strictly quasihypermetric, and by Corollary~\ref{2.4}
(see also Theorem~5.6 of~\cite{NW3}), we have $\mbar(X) < \infty$.
\end{example}

We can sum up our findings so far
on the finiteness of $\mbar(X)$ as follows.

\begin{theorem}
\label{2.17}
Let $\Xd$ be a compact metric space.
\begin{enumerate}
\item[(1)]
If $X$ is not quasihypermetric, then $\mbar(X) = \infty$.
\item[(2)]
If $X$ is quasihypermetric but not strictly quasihypermetric,
then $\mbar(X) < \infty$ and $\mbar(X) = \infty$ are both possible.
\item[(3)]
If $X$ is strictly quasihypermetric, then $\mbar(X) < \infty$ and
$\mbar(X) = \infty$ are both possible.
\end{enumerate}
\end{theorem}

\begin{proof}
Assertion~(1) follows from Theorem~\nwrefA{2.6};
assertion~(2) follows from Example~\ref{4ptexample}
and Theorem~\ref{2.9}; and
assertion~(3) follows from Corollary~\ref{2.3}
and Theorem~\ref{sqhmmbarinf}.
\end{proof}

We conclude with some remarks on the quasihypermetric property
and the strict quasihypermetric property.

For a compact metric space $\Xd$, the quasihypermetric property
is defined as a condition on the finite subsets of~$X$,
although by Theorem~\nwrefA{qhmconds} the property can also
be characterised measure-theoretically.
In particular, $X$ is quasihypermetric if and only if
every finite subset of~$X$ is quasihypermetric.
The next result implies that the strict quasihypermetric property
cannot be expressed as a condition on finite subsets.

\begin{theorem}
\label{sqhmexample}
There exists an infinite compact metric space
all of whose proper compact subsets
\textup{(}and its finite subsets in particular\textup{)}
are strictly quasihypermetric but which is not itself
strictly quasihypermetric.
\end{theorem}

\begin{proof}
Let $X$ and~$Y$ be copies of the unit circle $S^1$ in the plane,
with the euclidean metric.
Note that $X$ and~$Y$ are strictly quasihypermetric,
that normalised uniform measure on $X$ and~$Y$ is invariant,
and that therefore, by Theorem~\ref{2.2}, this measure is the unique
maximal measure on $X$ and~$Y$, and $\mbar(X) = \mbar(Y) < \infty$.
Form a metric space~$Z$ using the mechanism of Theorem~\ref{2.1},
setting the distance between each $x \in X$ and each $y \in Y$ to be~$c$,
where $2c = \mbar(X) + \mbar(Y)$.
By Theorem~\ref{2.1} part~(3), $Z$ is not strictly quasihypermetric.

Let $Z'$ be any proper compact subset of~$Z$,
and write $Z' = X' \cup Y'$ for suitable compact subsets
$X' \subseteq X$ and $Y' \subseteq Y$, at least one of which is proper.
If either $X'$ or~$Y'$ has no maximal measure,
then Theorem~\ref{2.1} part~(3) implies immediately that $Z'$ is
strictly quasihypermetric. If $X'$ and~$Y'$ both have maximal measures,
assume without loss of generality that $X'$ is a proper subset of~$X$.
Suppose that $\mbar(X') = \mbar(X)$.
Then the maximal measure on~$X'$ is also a maximal measure on~$X$,
but is certainly not uniform measure, and this contradicts the uniqueness
given by Theorem~\ref{2.2}. Therefore, $\mbar(X') < \mbar(X)$.
But now $2c > \mbar(X') + \mbar(Y')$, so Theorem~\ref{2.1} part~(3)
implies again that $Z'$ is strictly quasihypermetric.
\end{proof}

\providecommand{\WileyBibTextsc}{}
\let\textsc\WileyBibTextsc
\providecommand{\othercit}{}
\providecommand{\jr}[1]{#1}
\providecommand{\etal}{~et~al.}


\begin{thebibliography}{[10]}

\bibitem{Alex1}
\textsc{R.~Alexander},
Two notes on metric geometry,
\jr{Proc. Amer. Math. Soc.} \textbf{64}(2), 317--320 (1977).

\bibitem{AandS}
\textsc{R.~Alexander} and  \textsc{K.\,B.~Stolarsky},
Extremal problems of distance geometry related to energy integrals,
\jr{Trans. Amer. Math. Soc.} \textbf{193}, 1--31 (1974).

\bibitem{Bjo}
\textsc{G.~Bj{\"o}rck},
Distributions of positive mass,
  which maximize a certain generalized energy integral,
\jr{Ark. Mat.} \textbf{3}, 255--269 (1956).

\bibitem{Bret}
\textsc{J.~Bretagnolle}, \textsc{D.~Dacunha-Castelle}
  and \textsc{J.-L.~Krivine},
Lois stables et espaces $L^p$,
\jr{Ann. Inst. H. Poincar\'e Sect. B (N.S.)} \textbf{2}, 231--259 (1965/1966).

\bibitem{FR1}
\textsc{B.~Farkas} and  \textsc{Sz.\,Gy.~R{\'e}v{\'e}sz},
Rendezvous numbers of metric spaces---a potential theoretic approach,
\jr{Arch. Math. (Basel)} \textbf{86}(3), 268--281 (2006).

\bibitem{Hin1}
\textsc{A.~Hinrichs},
Averaging distances in finite-dimensional normed spaces
  and John's ellipsoid,
\jr{Proc. Amer. Math. Soc.} \textbf{130}(2), 579--584 (2002).

\bibitem{HKM}
\textsc{P.\,G.~Hjorth}, \textsc{S.\,L.~Kokkendorff}
  and \textsc{S.~Markvorsen},
Hyperbolic spaces are of strictly negative type,
\jr{Proc. Amer. Math. Soc.} \textbf{130}, 175--181 (2002).

\bibitem{NW1}
 \textsc{P.~Nickolas} and  \textsc{R.~Wolf},
Distance geometry in quasihypermetric spaces.\,{I},
\jr{Bull. Aust. Math. Soc.}, to appear.

\bibitem{NW3}
 \textsc{P.~Nickolas} and  \textsc{R.~Wolf},
Distance geometry in quasihypermetric spaces.\,{III},
submitted.

\bibitem{Wol1}
\textsc{R.~Wolf},
On the average distance property and certain energy integrals,
\jr{Ark. Mat.} \textbf{35}(2), 387--400 (1997).

\end{thebibliography}
\end{document}